\documentclass[a4paper,11pt]{article}

\usepackage{amsfonts}
\usepackage{amssymb}
\usepackage{amsthm}
\usepackage{amsmath}
\usepackage{a4wide}
\usepackage{mathrsfs}
\usepackage{epsfig}
\usepackage{palatino}

\newcommand{\pdfgraphics}{\ifpdf\DeclareGraphicsExtensions{.pdf,.jpg}\else\fi}
\usepackage{graphicx}   

\usepackage{color}
\definecolor{citegreen}{rgb}{0,0.6,0}
\definecolor{refred}{rgb}{0.8,0,0}
\usepackage[colorlinks, citecolor=citegreen,linkcolor=refred]{hyperref}

\numberwithin{equation}{section}

\theoremstyle{plain}

\newtheorem{teo}{Theorem}[section]
\newtheorem{lemma}[teo]{Lemma}
\newtheorem{prop}[teo]{Proposition}
\newtheorem{cor}[teo]{Corollary}
\newtheorem{ackn}{Acknowledgements\!}

\theoremstyle{definition}
\newtheorem{dfnz}[teo]{Definition}

\theoremstyle{remark}
\newtheorem{Remark}[teo]{Remark}

\numberwithin{equation}{section}

\def\eps{\varepsilon}
\def\HH{\mathcal H}

\def\loc{_{\operatorname{loc}}}
\def\NN{\mathbb N}
\newcommand{\Om}        {\Omega}

\def\R{\mathbb R}

\newcommand{\e }{\varepsilon }
\newcommand{\g }{\gamma}

\renewcommand{\l }{\lambda }

\newcommand{\s }{\sigma }

\renewcommand{\t }{\tau }

\renewcommand{\O }{\Omega }

\newcommand{\intbar}{\etaathop{\int\etaakebox(-13.5,0){\rule[4pt]{.7em}{0.3pt}}
\kern-6pt}\nolimits}

\newcommand{\be}{\begin{equation}}
\newcommand{\ee}{\end{equation}}
\newcommand{\bea}{\begin{equation*}}
\newcommand{\eea}{\end{equation*}}
\newcommand{\op}{\langle}
\newcommand{\cl}{\rangle}

\newcommand{\p}{\varphi}

\def\loc{_{\operatorname{\rm loc}}}

\def\R{{{\mathbb R}}}

\def\NN{{{\mathbb N}}}

\def\eps{\varepsilon}
\def\RRR{{\mathrm R}}

\def\dert{\partial_t}
\def\ders{\partial_s}
\def\TTT{{\mathbb{T}}}
\def\TT{\mathbb T}
\def\tt{\mathfrak t}

\def\RRR{{\mathrm R}}

\def\ress  {\begin{picture}(9,8)
\put (1,0){\line(0,1){6}}
\put (1,0){\line(1,0){5}}
\end{picture}}
\def\res{\,\ress}

\def\be{\begin{equation}}
\def\ee{\end{equation}}
\def\bea{\begin{eqnarray*}}
\def\bean{\begin{eqnarray}}
\def\eean{\end{eqnarray}}
\def\eea{\end{eqnarray*}}

\begin{document}
\pdfgraphics 

\title{Motion by Curvature of Planar Networks II}

\author{{Annibale Magni}\thanks{Mathematisches Institut, 
				Albert--Ludwigs--Universit\"at,
				Eckerstrasse 1, D--79104 Freiburg, Germany
        \texttt{annibale.magni@math.uni-freiburg.de}}\,,
        {Carlo Mantegazza}\thanks{Scuola Normale Superiore, 
        Piazza dei Cavalieri 7, 56126 Pisa, Italy
        \texttt{c.mantegazza@sns.it}}\,,        
        {Matteo Novaga}
        \thanks{Dipartimento di Matematica, Universit\`a di Pisa,
        Largo Bruno Pontecorvo 5, 56127 Pisa, Italy,
        \texttt{novaga@dm.unipi.it}}
}

\date{\today}

\maketitle

\begin{abstract}
We prove that the curvature flow of an embedded planar network of
three curves connected through a triple junction, with fixed endpoints
on the boundary of a given strictly convex domain, exists smooth as long 
as the lengths of the three curves stay far from zero. If this is the case for all times, then the evolution exists for all times and the network converges to
the Steiner minimal connection between the three endpoints.
\end{abstract}

\section{Introduction}

We are interested in the long time behavior of the evolution by
curvature of a {\em triod}, that is, a network of three planar 
curves meeting at a common point (called {\it triple junction}) with
equal angles (the so-called {\it Herring condition}) and 
with fixed endpoints on the boundary of a given convex domain in the Euclidean plane.

As for the mean curvature flow, this evolution can be regarded as the
gradient flow of the {\it Length} functional. 
In this respect, the Herring condition naturally arises from
the variational interpretation of the flow and 
corresponds to the local stability of the triple junction.

An important motivation for this study is due to the appearance of
this evolution in several models of materials science for the 
motion of grain boundaries in a polycrystalline material or, more
generally, of two--dimensional multiple phase systems
(see~\cite{hermul,gurtin,gurtin2} and references therein). Another
more theoretical motivation comes from the fact that this is possibly 
the simplest evolution by curvature of a nonsmooth set. Indeed, while the
mean curvature flow of a smooth submanifold is deeply, even if not 
completely, understood, the evolution of generalized submanifolds admitting singularities, 
for instance a {\it varifold}, has not been studied too much in 
detail after the seminal work by K.~Brakke~\cite{brakke} (see
also~\cite{caraballo1,degio7} for an alternative approach based on 
an implicit variational scheme introduced by J.~Almgren, J.~Taylor
and L.~Wang in~\cite{altawa} and, independently, by S.~Luckhaus and
T.~Sturzenhecker in~\cite{luckstur}), we
mention anyway the works of T.~Ilmanen~\cite{Il:93g} and 
K.~Kasai and Y.~Tonegawa~\cite{kaston}.

The mathematical analysis of this flow started
in~\cite{bronsard} (see also~\cite{kinderliu}), where short time 
existence and uniqueness of a smooth flow has been established, and continued
in~\cite{mannovtor} where the authors proved that, at the first 
singular time, either the curvature blows--up or the length of one of the three curves goes to zero on a sequence of times. 
Extending the analysis performed by 
G.~Huisken for the mean curvature flow (see~\cite{huisk3} and
references therein), they could also rule out certain kinds of 
singularities, namely the so--called {\em Type I} singularities,
corresponding to a specific blow--up rate of the curvature at the 
singular time. A significant difficulty of this analysis is the lack
of maximum principle, due to the presence of the triple 
junction, which requires new arguments in order to estimate geometric
quantities such as the curvature and its derivatives.

In this paper we complete the program started in~\cite{mannovtor} and
we prove that no singularity can arise during the 
evolution of a triod, independently of the type of singularity. More
precisely, our main result is the following.

\begin{teo}\label{teomain}
For any smooth, embedded, initial triod $\TT_0$ in a 
strictly convex set $\Omega\subset\R^2$, with fixed endpoints $P^1,
P^2, P^3\in\partial\Omega$, there exists a unique smooth
evolution by curvature of $\TT_0$ which at every time is a
nondegenerate smooth embedded triod in 
$\Omega$, in a maximal time interval
$[0,T)$. Then, either the inferior limit of the length of one of three curves of the
triod $\TT_t$ goes to zero as $t\to T$, or $T=+\infty$ and $\TT_t$
tends as $t\to +\infty$ to the unique Steiner triod connecting the
three fixed endpoints.
\end{teo}

Our strategy is based on the analysis of the blow--up of the flow at a
given point, independently of the behavior of the 
curvature. Using some ideas presented in~\cite{ilman3} (see
also~\cite{manmaggra}), which are based on Huisken's {\em monotonicity 
formula} (see~\cite{huisk3}), we are able to classify all the possible
blow--up limits. It turns out that the only admissible 
configurations are a straight line, a halfline or a flat unbounded
triod (see Proposition~\ref{resclimit}). As none of them arises 
from a singular point of the flow, we obtain our main result. A fundamental 
ingredient in our analysis are the interior regularity estimates of K.~Ecker and G.~Huisken (see~\cite{eckhui1}), 
which we combine with the estimates 
on the curvature and its derivatives obtained in~\cite{mannovtor}.

One difficulty in this classification is to show that the possible
limits necessarily have multiplicity one. This follows from a 
geometric argument proposed in~\cite{hamilton3} (see also~\cite{huisk2})
and extended in~\cite{mannovtor} to the case of a triod, consisting in estimating 
from below a kind of ``embeddedness measure'', which is strictly positive when no self--intersections are present and showing that 
it is monotonically increasing for an evolving triod. We underline that 
it is not clear to us how to obtain a similar bound for a general
network (with multiple triple junctions), since the analogous quantity is no longer 
monotone if there are more than two triple junctions.

\medskip

Recently, T.~Ilmanen, A.~Neves and F.~Schulze announced a
comprehensive analysis of the evolution by curvature 
of a general network with several 
{\em multiple} (not only triple) junctions, 
with any angles between the concurring edges. This would clearly
include and greatly generalize our work. 

In the preliminary paper~\cite{Ilnevsch}, the authors 
prove a local regularity result stating that, if 
the Gaussian density is bounded away from two and the network has no loops in an open set
$A$, then the evolution of the network is smooth in $A$.

Independently of such a result, this paper deals with the 
simpler situation of a single 
triod with fixed endpoints, in a strictly convex domain. Our goal is 
simply to show that singularities
cannot happen in this special case, hence completing the program
started in~\cite{mannovtor}. 
We point out that our method cannot be directly extended 
to the case of a network with more than two triple 
junctions, due to the aforementioned main difficulty in showing that the  
blow--ups at the singular points have multiplicity one.

\medskip

\begin{ackn} Annibale~Magni was partially supported by
the DFG Collaborative Research Center SFB/Transregio 71.\\
Carlo~Mantegazza and Matteo~Novaga were partially supported by
  the Italian FIRB Ideas ``Analysis and Beyond''.
\end{ackn}

\section{Definitions and Preliminary Results} 

\begin{dfnz}
Let $\O \in \R^2$ be a smooth open set and $\TT = \cup^3_{i=1}
\s^i$ the union of three embedded (at least $C^2$), regular
(i.e. $\s_x \neq 0$ for all $x \in [0,1]$) curves $\s^i : [0,1]
\rightarrow \overline{\O}$. Let $P^i \in \partial
\O$, for $i \in \{1,2,3\}$, three distinct points. We say that $\TT$ is a
{\it triod} in $\Omega$ if

\begin{itemize}
\item $\sigma^i(x)\in\partial\Omega$ if and only if $x=1$, for all $i
  \in \{1,2,3\}$;
\item $O = \s^i(0)$ for all $i \in \{1,2,3\}$;
\item $\s^i(x)=\s^j(y)$ for $i,j\in\{1,2,3\}$ and $x,y\in[0,1]$ if and
  only if $x=y=0$ or $i=j$ and $x=y$;
\item $\s^i(1) =  P^i$ for all $i \in \{1,2,3\}$;
\item $\sum^3_{i=1}\frac{\s^i_x(0)}{|\s^i_x(0)|} = 0$. 
\end{itemize}
Under these conditions, we will call $O$ the {\em 3--point of the
  triod} $\TT$ and $P^i$ the  {\em endpoints of the triod} $\TT$.
\end{dfnz}

For a given ``initial'' triod $\TT_0 = \cup^3_{i=1} \s^i$, we consider
the following motion by curvature (see~\cite{bronsard} and~\cite{mannovtor}).

\begin{dfnz}
We say that the one parameter family of triods $\TT_t = \cup^3_{i=1}
\g(\cdot, t)$ evolve by curvature (staying embedded) in the time interval
$[0,T)$ ($T > 0$), if the three family of curves $\g^i : [0,1] \times
[0,T) \rightarrow \overline{\O}$ are at least of class $C^2$ in the first variable and
of class $C^1$ in the second one, and satisfy the following quasilinear
parabolic system,
\begin{equation}\label{problema}
\begin{cases}
\begin{array}{lll}
\gamma_x^i(x,t)\not=0\qquad &&\text{ regularity}\\
\gamma^i(x,t)\in\partial\Omega\qquad &\text{{ iff}
  $x=1$}&\text{intersection with $\partial\Omega$ only at the
  endpoints}\\
\gamma^i(x,t)\not=\gamma^i(y,t)\qquad &\text{{ if} $x\not=y$}&\text{ simplicity}\\
\gamma^i(x,t)=\gamma^j(y,t)\,\Leftrightarrow\,x,y=0\qquad&\text{{    if} $i\not=j$}
&\text{ intersection only  at the 3--point}\\
\sum_{i=1}^3\frac{\gamma_x^i(0,t)}{\vert{\gamma_x^i(0,t)}\vert}=0
\qquad&&\text{ angles of $120$ degrees at the 3--point}\\
\gamma^i(1,t)=P^i
\qquad &&\text{ fixed endpoints condition}\\
\gamma^i(x,0)=\sigma^i(x)\qquad &&\text{ initial data}\\
\gamma^i_t(x,t)=\frac{\gamma_{xx}^i(x,t)}{{\vert{\gamma_x^i(x,t)}\vert}^2}\qquad
&&\text{ motion by curvature}
\end{array}
\end{cases}
\end{equation}
for every $x\in[0,1]$, $t\in[0,T)$ and $i, j\in\{1, 2, 3\}$.
\end{dfnz}

To denote a flow we will often
write simply $\TTT_t$ instead of letting explicit the curves $\gamma^i$
which compose the triods.\\
Moreover, it will be also useful to describe a triod as a map
$F:\TTT\to\overline{\Omega}$ from a fixed {\em standard} triod $\TTT$
in $\R^2$,  composed of three unit segments from the origin in the
plane, forming angles of $120$ degrees. In this case we will still
denote with $O$ the 3--point of $\TTT$ and with $P^i$ the three
endpoints of such standard triod.\\
The evolution then will be given by a map
$F:\TTT\times[0,T)\to\overline{\Omega}$,
constructed naturally from  the curves
$\gamma^i$, so $\TTT_t=F(\TTT,t)$.

In~\cite{mannovtor} the following short time existence and uniqueness theorem has been proven.

\begin{teo}
For any smooth initial triod $\TT_0$ in a 
convex set $\Omega\subset\R^2$, there exists a unique smooth
solution of Problem~\eqref{problema} in a maximal time interval
$[0,T)$, with $T>0$. In particular, the evolving triod does not exit
the open set $\Omega$ (with exception of the three fixed 
endpoints $P^i$).
\end{teo}

The goal of this paper is to show the following result which, with the
above theorem, gives Theorem~\ref{teomain} in the introduction.

\begin{teo}\label{corfin}
Given a triod $F: \TT\times [0,T) \to \overline{\Omega}$ evolving by
curvature, where $\Omega$ is a strictly convex open subset of $\R^2$,
either the inferior limit of the 
length of one of three curves of the 
triod $\TT_t$ goes to zero as $t\to T$, or $T=+\infty$ and $\TT_t$
tends as $t\to +\infty$ to the unique Steiner triod connecting the
three fixed endpoints.
\end{teo}

We remark that the first situation can actually happen, for instance,
if the triangle formed by the points $P^1$, $P^2$, $P^3$ has one angle larger than 120
degrees. Notice that the strict convexity of $\Omega$ implies that
such triangle is nondegenerate.

Along the paper we will make extensive use of the following notation,
$$
\begin{array}{ll}
\tau^i=\tau^i(x,t)= \frac{\gamma_x^i}{\vert\gamma_x^i\vert}
&\qquad\text{{ unit tangent vector to} $\gamma^i$}\,,\\
\nu^i=\nu^i(x,t)= {\mathrm R} \tau^i(x,t)={\mathrm R} \frac{\gamma_x^i}{\vert\gamma_x^i\vert}
&\qquad\text{{ unit normal vector to} $\gamma^i$}\,,\\
O=O(t)=\gamma^i(0,t)
&\qquad\text{{ 3--point of } $\TTT_t$}\,,\\
\underline{v}^i=\underline{v}^i(x,t)= \frac{\gamma_{xx}^i}{{\vert{\gamma_x^i}\vert}^2}
&\qquad\text{{ velocity of the point} $\gamma^i(x,t)$}\,,\\
\lambda^i=\lambda^i(x,t)= \frac{\langle\gamma_{xx}^i\,\vert\,\tau^i\rangle}
{{\vert{\gamma_x^i}\vert}^2}=\frac{\langle\gamma_{xx}^i\,\vert\,\gamma^i_x\rangle}
{{\vert{\gamma_x^i}\vert}^3}
&\qquad\text{{ tangential velocity of the point} $\gamma^i(x,t)$}\,,\\
k^i=k^i(x,t)= \frac{\langle\gamma_{xx}^i\,\vert\,\nu^i\rangle}
{{\vert{\gamma_x^i}\vert}^2}=
\langle\partial_s\tau^i\,\vert\,\nu^i\rangle=
-\langle\partial_s \nu^i\,\vert\,\tau^i\rangle
&\qquad\text{{ curvature at the point} $\gamma^i(x,t)$}\,,
\end{array}
$$
where with $s$ we have denoted the arclength parameter on any of the
curves and with $\RRR : \R^2 \rightarrow \R^2$ the counterclockwise
rotation centered in the origin of $\R^2$ of angle
$\pi/2$. Furthermore, we set $\underline{\lambda}^i=\lambda^i\tau^i$
and $\underline{k}^i=k^i\nu^i$, from which it follows that
$\underline{v}^i=\underline{\lambda}^i+\underline{k}^i$ and
$|\underline{v}^i|^2=(\lambda^i)^2+(k^i)^2$. We will also denote by
$L^i$ the length of the $i$--th curve of the triod and by $L = L^1 +
L^2 + L^3$ its global length.

We now state some results which have been proven in~\cite{mannovtor}.

\begin{lemma}\label{commute}
If $\gamma$ is a curve of a triod moving by curvature, which means that
$$
\gamma_t=\frac{\gamma_{xx}}{|\gamma_x|^2}=\lambda\tau+k\nu\,,
$$
then the following commutation rule holds,
$$
\partial_t\partial_s=\partial_s\partial_t +
(k^2 -\lambda_s)\partial_s \,.
$$
\end{lemma}

With the help of Lemma~\ref{commute} one gets the following formulas.

\begin{lemma}\label{evoluzioni}
For any curve evolving by curvature, there holds
\begin{align*}
\dert\tau=&\,
\dert\ders\gamma=\ders\dert\gamma+(k^2-\lambda_s)\ders\gamma =
\ders(\lambda\tau+k\nu)+(k^2-\lambda_s)\tau =
(k_s+k\lambda)\nu\\
\dert\nu=&\, \dert({\mathrm R}\tau)={\mathrm
R}\,\dert\tau=-(k_s+k\lambda)\tau\\
\dert k=&\, \dert\langle \ders\tau , \nu\rangle=
\langle\dert\ders\tau , \nu\rangle
= \langle\ders\dert\tau , \nu\rangle +
(k^2-\lambda_s)\langle\ders\tau , \nu\rangle\\
=&\, \ders\langle\dert\tau , \nu\rangle + k^3-k\lambda_s =
\ders(k_s+k\lambda) + k^3-k\lambda_s\nonumber\\
=&\, k_{ss}+k_s\lambda + k^3\nonumber\\
\dert\lambda =&\, -\dert\partial_x\frac{1}{|\gamma_x|}=
\partial_x \frac{\langle\gamma_x , \gamma_{tx}\rangle}
{|\gamma_x|^3}=
\partial_x \frac{\langle\tau , \ders (\lambda\tau+k\nu)\rangle}
{|\gamma_x|}=\partial_x \frac{(\lambda_s - k^2)}
{|\gamma_x|}\\
=&\, \ders(\lambda_s - k^2) -\lambda(\lambda_s -
k^2)=\lambda_{ss} -\lambda\lambda_s - 2kk_s +\lambda k^2\,.\nonumber
\end{align*}
\end{lemma}

Taking into account the compatibility conditions at the 3--point, we have the
following lemma.

\begin{lemma}
At the 3--point of a triod $\TT_t$ evolving as in Problem~\eqref{problema} hold
\begin{gather*}
\lambda^i=\frac{k^{i-1}-k^{i+1}}{\sqrt{3}} \\
k^i=\frac{\lambda^{i+1}-\lambda^{i-1}}{\sqrt{3}}
\end{gather*}
where the indices are understood modulo three. Moreover
\begin{equation*}
\sum_{i=1}^3 k^i=\sum_{i=1}^3\lambda^i=0
\end{equation*}
$$
k_s^i+\lambda^i k^i=k_s^j+\lambda^j k^j
$$
for every pair $i, j \in \{1,2,3\}$.
\end{lemma}

The key theorem for the analysis of the singularities is the following
result~\cite[Theorem~3.18]{mannovtor}.

\begin{prop}\label{curvexplod}
If $[0,T)$ is the maximal time interval of existence of a smooth
solution $\TTT_t$ with $T<+\infty$ of Problem~\eqref{problema},
then one of the following possibilities holds:
\begin{itemize}
\item the inferior limit of the length of one curve of $\TTT_t$ tends to zero as $t\to T$,
\item $\limsup_{t\to T}\int_{\TTT_t}k^2\,ds=+\infty$.
\end{itemize}
Moreover, if the lengths of the three curves are uniformly
bounded away from zero, then the superior limit is actually a limit.
\end{prop}

In the next section we will show that if the lengths of the three curves are uniformly
bounded away from zero, no singularity can develop. We
now introduce the tools and the estimates which we will need.

Let $F : \TT\times [0,T)\to\mathbb{R}^2$ a curvature flow for a
triod in its maximal time interval of existence, then a 
modified form of Huisken's monotonicity formula holds.\\
Let $x_0\in\Om$ and $\rho_{x_0}:\R^2\times[0,T)$ be the {\em backward
  heat kernel} of $\R^2$ relative to $(x_0,T)$, that is
$$
\rho_{x_0}(x,t)=\frac{e^{-\frac{\vert
    x-x_0\vert^2}{4(T-t)}}}{\sqrt{4\pi(T-t)}}\,.
$$

\begin{prop}[Monotonicity Formula -- Proposition~6.4 in~\cite{mannovtor}]
\label{promono}
For every $x_0\in\R^2$ and $t\in [0,T)$ the following identity holds
\begin{align}
\frac{d\,}{dt}\int_{{\TT_t}} \rho_{x_0}(x,t)\,ds=
&\,-\int_{{\TT_t}} \left\vert\,\underline{k}+\frac{(x-x_0)^{\perp}}{2(T-t)}\right\vert^2
\rho_{x_0}(x,t)\,ds\label{eqmonfor}\\
&\,+\sum_{i=1}^3\left\langle\,\frac{P^i-x_0}{2(T-t)} ,\tau^i(1,t)\right\rangle
\rho_{x_0}(P^i,t)\,.\nonumber
\end{align}
Integrating between $t_1$ and $t_2$ with $0\leq t_1\leq t_2<T$ we get
\begin{align}
\int_{t_1}^{t_2}\int_{{\TT_t}}
\left\vert\,\underline{k}+\frac{(x-x_0)^{\perp}}{2(T-t)}\right\vert^2
\rho_{x_0}(x,t)\,ds\,dt = 
&\,\int_{{\TT_{t_1}}} \rho_{x_0}(x,t_1)\,ds -
\int_{{\TT_{t_2}}} \rho_{x_0}(x,t_2)\,ds\nonumber\\
&\,+\sum_{i=1}^3\int_{t_1}^{t_2}
\left\langle\,\frac{P^i-x_0}{2(T-t)} , \tau^i(1,t)\right\rangle
\rho_{x_0}(P^i,t)\,dt\,.\nonumber
\end{align}
\end{prop}

\begin{Remark}Notice that the monotonicity formula for a triod differs
  from the standard one because of a boundary term. Thanks to the next
  lemma, this extra term will not change to a big extent the blow--up
  analysis for the curvature motion of triods.
\end{Remark}

\begin{lemma}[Lemma~6.5 in~\cite{mannovtor}]\label{stimadib} 
Setting $\vert P^i- x_0\vert=d^i$, 
for every index $i\in\{1,2,3\}$ the following estimate holds 
\begin{equation*}
\left\vert\int_{t}^{T}\left\langle\,\frac{P^i-x_0}{2(T-\xi)} , 
\tau^i(1,\xi)\right\rangle\rho_{x_0}(P^i,\xi)\,d\xi\,\right\vert\leq
\frac{1}{\sqrt{2\pi}}\int\limits_{d^i/\sqrt{2(T-t)}}^{+\infty}e^{-y^2/2}\,dy\leq
1/2\,.
\end{equation*}
As a consequence, for every point $x_0\in\R^2$, we have
\begin{equation*}
\lim_{t\to T}\sum_{i=1}^3\int_{t}^{T}
\left\langle\,\frac{P^i-x_0}{2(T-\xi)} , \tau^i(1,\xi)\right\rangle
\rho_{x_0}(P^i,\xi)\,d\xi=0\,.
\end{equation*}
\end{lemma}

\begin{prop}
If for every $x_0\in\R^2$ we define the functions $\Theta:\TT\times[0,T)\to\R$ as
$$
\Theta(x_0,t)=\int_{\TTT_t}\rho_{x_0}(x,t)\,ds\,,
$$
then, the limit
$$
\widehat{\Theta}(x_0)=\lim_{t\to T}\Theta(x_0,t)=\lim_{t\to T}\int_{\TTT_t}\rho_{x_0}(x,t)\,ds
$$
exists and it is finite.\\
Moreover, the map $\widehat{\Theta}:\mathbb{R}^2\to\mathbb{R}$ is upper
semicontinuous.
\end{prop}
\begin{proof}
We consider the function $b:\R^2\times [0,T)\to\R$ given by
$$
b(x_0,t)=\int_t^T\sum_{i=1}^3
\left\langle\,\frac{P^i-x_0}{2(T-\xi)}\,\biggl\vert\,\tau^i(1,\xi)\right\rangle
\rho_{x_0}(P^i,\xi)\,d\xi\,.
$$
Lemma~\ref{stimadib} says that $b$ is uniformly bounded and for every
$x_0\in\R^2$ we have $\lim_{t\to T}b(x_0,t)=0$. Hence, the monotonicity
formula~\eqref{eqmonfor} can be rewritten as
$$
\frac{d\,}{dt}\bigl(\Theta(x_0,t)+b(x_0,t)\bigr)=-\int_{{\TTT_t}}
\left\vert\,\underline{k}+\frac{(x-x_0\,)^{\perp}}{2(T-t)}\right\vert^2
\rho_{x_0}(x,t)\,ds\leq 0\,,
$$
hence, being nonincreasing and bounded from below, the
functions $\bigl(\Theta(\cdot , t)+b(\cdot , t)\bigr)$ pointwise converge on all
$\R^2$ when $t\to T$. Since we have seen that $b(\cdot , t)$ 
pointwise converge to zero everywhere, the limit $\widehat{\Theta}(x_0)$
exists for every $x_0\in\R^2$.\\
As the convergence of the continuous functions $\bigl(\Theta(\cdot ,
t)+b(\cdot , t)\bigr)$ to $\widehat{\Theta}:\R^2\to\R$ is monotone
nonincreasing, this latter is upper semicontinuous.
\end{proof}

We now introduce the rescaling procedure of
Huisken~\cite{huisk3}.\\
For a fixed $x_0\in\R^2$, let $\widetilde{F}_{x_0}:\TTT\times
[-1/2\log{T},+\infty)\to\R^2$ be the map
$$
\widetilde{F}_{x_0}(p,\tt)=\frac{F(p,t(\tt))-x_0}{\sqrt{2(T-t(\tt))}}\qquad
\tt(t)=-\frac{1}{2}\log{(T-t)} \,.
$$
Then the rescaled triods are given by 
$$
\widetilde{\TTT}_{x_0,\tt}=\frac{\TTT_{t(\tt)}-x_0}{\sqrt{2(T-t(\tt))}}
$$
and they evolve according to the equation 
$$
\frac{\partial\,}{\partial
  \tt}\widetilde{F}_{x_0}(p,\tt)=\widetilde{\underline{v}}(p,\tt)+\widetilde{F}_{x_0}(p,\tt) \,,
$$
where 
$$
\widetilde{\underline{v}}(p,\tt)=\frac{\underline{v}(p,t(\tt))}{\sqrt{2(T-t(\tt))}}=
\widetilde{\underline{k}}+\widetilde{\underline{\lambda}}=
\widetilde{k}\nu+\widetilde{\lambda}\tau\qquad \text{ and }\qquad
t(\tt)=T-e^{-2\tt}\,.
$$
Notice that we did not put the ``tilde'' over the unit tangent and
normal, since they do not change under rescaling.\\
We will often write $\widetilde{O}(\tt)=\widetilde{F}_{x_0}(0,\tt)$ 
for the 3--point of the rescaled triod $\widetilde{\TTT}_{x_0,\tt}$,
when there is no ambiguity on the point $x_0$.\\
The rescaled curvature evolves according to the following equation
\begin{equation*}
{\partial_\tt} \widetilde{k}=
\widetilde{k}_{\sigma\sigma}+\widetilde{k}_\sigma\widetilde{\lambda} +
\widetilde{k}^3 -\widetilde{k} \,,
\end{equation*}
which can be obtained by means of the
commutation law
\begin{equation*}
{\partial_\tt}{\partial_\sigma}={\partial_\sigma}{\partial_\tt} + (\widetilde{k}^2
-\widetilde{\lambda}_\sigma-1){\partial_\sigma}\,,
\end{equation*}
where we denoted with $\sigma$ the arclength parameter for
$\widetilde{\TTT}_{x_0,\tt}$.

By a straightforward computation
(see~\cite{huisk3} and~\cite[Lemma~6.7]{mannovtor}) we have the following
rescaled version of the monotonicity formula.

\begin{prop}[Rescaled Monotonicity Formula]
Let $x_0\in\R^2$ and set  
$$
\widetilde{\rho}(x)=e^{-\frac{\vert x\vert^2}{2}}
$$
For every $\tt\in[-1/2\log{T},+\infty)$ the following identity holds
\begin{equation*}
\frac{d\,}{d\tt}\int_{\widetilde{\TTT}_{x_0,\tt}}
\widetilde{\rho}(x)\,d\sigma=
-\int_{\widetilde{\TTT}_{x_0,\tt}}\vert
\,\widetilde{\underline{k}}+x^\perp\vert^2\widetilde{\rho}(x)\,d\sigma
+\sum_{i=1}^3\left\langle\,{\widetilde{P}^i_{x_0,\tt}}
\,\Bigl\vert\,{\tau}^i(1,t(\tt))\right\rangle
\widetilde{\rho}(\widetilde{P}^i_{x_0,\tt})
\end{equation*}
where $\widetilde{P}^i_{x_0,\tt}=\frac{P^i-x_0}{\sqrt{2(T-t(\tt))}}$.\\  
Integrating between $\tt_1$ and $\tt_2$ with 
$-1/2\log{T}\leq \tt_1\leq \tt_2<+\infty$ we get
\begin{align} 
\int_{\tt_1}^{\tt_2}\int_{\widetilde{\TTT}_{x_0,\tt}}\vert
\,\widetilde{\underline{k}}+x^\perp\vert^2\widetilde{\rho}(x)\,d\sigma\,d\tt=
&\, \int_{\widetilde{\TTT}_{x_0,\tt_1}}\widetilde{\rho}(x)\,d\sigma - 
\int_{\widetilde{\TTT}_{x_0,\tt_2}}\widetilde{\rho}(x)\,d\sigma\label{reseqmonfor-int}\\
&\,+\sum_{i=1}^3\int_{\tt_1}^{\tt_2}\left\langle\,{\widetilde{P}^i_{x_0,\tt}}\,
\Bigl\vert\,{\tau}^i(1,t(\tt))\right
\rangle\widetilde{\rho}(\widetilde{P}^i_{x_0,\tt})\,d\tt\,.\nonumber
\end{align}
\end{prop}

Then, we have the analog of Lemma~\ref{stimadib}.

\begin{lemma}[Lemma~6.8 in~\cite{mannovtor}]\label{rescstimadib} 
For every index $i\in\{1,2,3\}$ 
the following estimate holds 
\begin{equation*}
\left\vert\int_{\tt}^{+\infty}\left\langle\,{\widetilde{P}^i_{x_0,\xi}}\,
\Bigl\vert\,{\tau}^i(1,t(\xi))\right
\rangle\widetilde{\rho}(\widetilde{P}^i_{x_0,\xi})\,d\xi\right\vert\leq \sqrt{\pi/2}\,.
\end{equation*}
Then, for every $x_0\in\R^2$,
\begin{equation*}
\lim_{\tt\to  +\infty}\sum_{i=1}^3\int_{\tt}^{+\infty}\left\langle\,{\widetilde{P}^i_{x_0,\xi}}\,
\Bigl\vert\,{\tau}^i(1,t(\xi))\right
\rangle\widetilde{\rho}(\widetilde{P}^i_{x_0,\xi})\,d\xi=0\,.
\end{equation*}
\end{lemma}

Before showing the key proposition about the blow-up limits of the
flow at a singularity, we need some technical lemmas.
 
\begin{lemma}[Second statement in Lemma~6.10 in~\cite{mannovtor}]\label{rescestim}
For every ball $B_R$ centered at the
origin of $\R^2$, we have the following
estimates with a constant $C_R$ independent of 
$x_0\in\R^2$ and $\tt\in[-1/2\log{T},+\infty)$
$$
\HH^1({\widetilde{\TTT}_{x_0,\tt}}\cap B_R)\leq C_R\,.
$$
\end{lemma}

\begin{dfnz} We say that a sequence of triods converges in the
  $C^r\loc$ topology if, after reparametrizing all their curves with
  the arclength, they converge in $C^r$ on every compact set of $\R^2$.\\
The definition of convergence in $W^{n,p}\loc$ is analogous.
\end{dfnz}

Given the smooth flow $\TTT_t=F(\TTT,t)$, we consider two points $p=F(x,t)$ and
$q=F(y,t)$ belonging to $\TTT_t$ 
and we define $\Gamma_{p,q}$ to be the {\em geodesic}
curve contained in $\TTT_t$ connecting $p$ and $q$. Then we let 
$A_{p,q}$ to be the area of the open region ${\mathcal A}_{p,q}$ in $\R^2$ enclosed
by the segment  $[p,q]$ and the curve $\Gamma_{p,q}$, as in the figure.

\begin{figure}[h]
\begin{center}
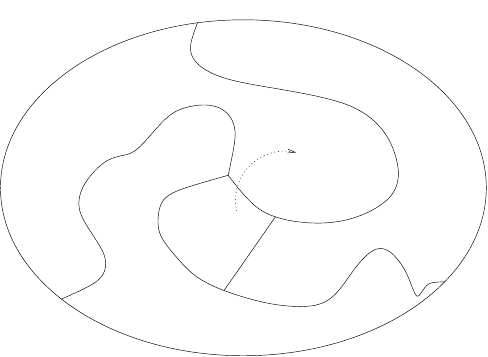
\end{center}
\end{figure}

If the region ${\mathcal A}_{p,q}$ is not connected, we let
$A_{p,q}$ to be the sum of the areas of its connected components.

We consider the function 
$\Phi_t: \TTT\times\TTT\to \R\cup\{+\infty\}$ as 
\begin{equation*}
\Phi_t(x,y)=
\begin{cases}
\frac{\vert p-q\vert^2}{A_{p,q}}\qquad &\text{ { if} $x\not=y$,}\\
4\sqrt{3}\qquad  & \text{ { if} $x$ and $y$ coincide with the 3--point
  $O$ of $\TTT$,}\\
+\infty\qquad  & \text{ { if}  $x=y\not=O$}
\end{cases}
\end{equation*}
\\
Since $\TTT_t$ is smooth and the $120$ degrees condition holds, it is
easy to check that $\Phi_t$ is a lower semicontinuous function. Hence,
by the compactness of $\TTT$, the following infimum is actually a
minimum
\begin{equation*}
E(t) = \inf_{x,y\in\TTT}\Phi_t(x,y)
\end{equation*}
for every $t\in[0,T)$.\\

If the triod $\TTT_t$ has no self--intersections we have
$E(t)>0$, the converse is clearly also true.\\
Moreover, $E(t)\leq\Phi_t(0,0)=4\sqrt{3}$ always holds, thus when
$E(t)>0$ the two points $(p,q)$ of a minimizing pair $(x,y)$ can coincide if
and only if $p=q=O$.\\
Eventually, since the evolution is smooth, it is easy to see that the
function $E:[0,T)\to\R$ is continuous.

\begin{prop}[Theorem~4.6 in~\cite{mannovtor}]\label{dlteo} If $\Omega$
  is strictly convex, 
there exists  a constant $C>0$ depending only on $\TTT_0$ 
such that $E(t)>C>0$ for every $t\in[0,T)$.\\
Hence, the triods $\TTT_t$ remain embedded in all the maximal interval of 
existence of the flow.
\end{prop}

\begin{lemma}\label{Elemma1}
If $\Omega$ is strictly convex, the function
\begin{equation*}
E(\TTT)= \inf_{\genfrac{}{}{0pt}{}
{p,q\in\TTT}{p\not=q}}\frac{\vert p-q\vert^2}{A_{p,q}}\,,
\end{equation*}
defined on the class of $C^1$ triods without self--intersections 
(bounded or unbounded and with or without end points or 3--points), is upper
semicontinuous with respect to the $C^1\loc$ convergence.\\
Moreover, $E$ is {\em dilation and translation invariant}.
Consequently, every $C^1\loc$ limit ${\TTT}_\infty$ of a sequence of rescaled triods
$\widetilde{\TTT}_{\tt}$ has no self--intersections,
has multiplicity one (outside the endpoints and the 3--point if
present), and satisfies $E(\TTT_\infty)>C>0$ where the 
constant $C$ is given by Proposition~\ref{dlteo}.
\end{lemma}

\begin{proof}
The dilation--translation invariance and the upper semicontinuity
of the function $E$ are straightforward, by
the $C^1\loc$ convergence. This latter obviously implies the final statement
of the theorem once we show the embeddedness and the multiplicity
one properties.

Suppose that a sequence of rescaled triods
$\widetilde{\TTT}_{\tt_j}$ converges to some limit $\TTT_\infty$. If
this latter has a transversal self--intersection, the triods of the
approximating sequence must definitively have self--intersections too,
but this contradicts Proposition~\ref{dlteo}. By the same reason,
$\TTT_\infty$ cannot contain loops 
and a self--intersection at the 3--point is impossible. 
Similarly, we can exclude a self--intersection at an endpoint.
Indeed, $\TTT_\infty$ can contain an endpoint $P^i_\infty$ only if we rescale 
the evolving triods around one of the endpoints $P^i$, hence
the convexity of $\Omega$ implies that
$\TTT_\infty$ lies in a halfspace and the
only way a self--intersection at $P^i_\infty$ can appear  
is that the limit curve starting at $P^i_\infty$ is tangent to another limit
piece of a curve of $\TTT_\infty$. 
As we assumed that the lengths of the curves are uniformly 
bounded below away from zero, then either $\TTT_\infty$ contains a loop (as it is connected), 
or a piece of the limit triod containing $P^i_\infty$ has multiplicity two,
coming from the ``collapsing'' of two pieces of curve in the sequence of rescaled triods. 
We saw that the first case is impossible, then the
second one is excluded by the argument below, which deals with the 
multiplicity of the limit set.\\
The only other possible self--intersections of the limit set can happen
at self--tangency points. 
By the $C^1\loc$ convergence, in a sufficiently small ball of radius $R$ around any
of such points $x\in\R^2$, definitively, for every rescaled triod
$\widetilde{\TTT}_{\tt_j}$, there 
must be some number of curves which are ``pieces'' of
$\widetilde{\TTT}_{\tt_j}$, such that 
they are all disjoint, all graphs on the tangent
line $L$ to $\TTT_\infty$ at $x$ and all converging to the same limit
$C^1$ graph $\TTT_\infty\cap B_R$. Considering two of such pieces of
curves, say $\sigma^1_j$ and $\sigma^2_j$, we 
take the point $p_j$ and $q_j$ which are the intersections of the
orthogonal line to $L$ at $x$ and the two curves. By hypotheses,
the distance $d_j$ between $p_j$ and $q_j$ goes to zero. Moreover, as every
rescaled triod is connected there must be a geodesic curve (in the entire rescaled triod) connecting
such two points. This means that the open region ${\mathcal A}_{p_j,q_j}$ is well defined and its area $A_{p_j,q_j}$ is larger than the area $S_j$ contained between the two curves--graphs $\sigma^1_j$ and
$\sigma^2_j$ in $B_R$. Hence we get $E(\widetilde{\TTT}_{\tt_j})\leq d_j^2/S_j$. If we now rescale the ball $B_R$ by a factor $1/d_j$, 
the two curves $\sigma^1_j,\,\sigma^2_j$ 
converge, as $j\to\infty$, to two straight lines
parallel to $L$. As the distance between the rescaled of the points
$p_j$, $q_j$ is one, the distance between 
this two straight lines is also one. 
As the function $E$ is dilation invariant
and the rescaling of the region between the
two curves in $B_R$ converges to a half--strip in $\R^2$, we conclude that 
$\lim_{j\to\infty} S_j/d_j^2=+\infty$, hence
$E(\widetilde{\TTT}_{\tt_j})\to 0$ which is a contradiction, by
Proposition~\ref{dlteo}.
\end{proof}

We can now show the following result, which is analogous to the
(stronger) one for Type I singularities proved
in~\cite[Proposition~6.16]{mannovtor}.

\begin{prop}\label{resclimit}
Assume that the lengths of the three curves of the triods $\TT_t$ are
uniformly bounded away from zero during the evolution.

For every $x_0\in\R^2$ and every subset $\mathcal I$ of  $[-1/2\log
T,+\infty)$ with infinite Lebesgue measure, 
there exists a sequence of rescaled times $\tt_j\to+\infty$, with $\tt_j\in{\mathcal I}$, such that the sequence of rescaled triods
$\widetilde{\TTT}_{x_0,\tt_{j}}$ converges in the $C^1\loc$ topology
to a limit set 
$\TTT_\infty$ which, {\em if not empty}, is one of the following:
\begin{itemize}
\item a straight line through the origin with multiplicity one (in this case $\widehat\Theta(x_0)=1$);
\item an infinite flat triod centered at the origin with multiplicity
  one, except its 3--point (in this case $\widehat\Theta(x_0)=3/2$);
\item a halfline from the origin of multiplicity one, except the origin (in this case $\widehat\Theta(x_0)=1/2$).
\end{itemize}
Moreover, the $L^2$ norm of the curvature in every ball $B_R\in\R^2$
along such sequence goes to zero, as $j\to\infty$.

For every sequence of rescaled triods $\TTT_{x_0,\tt_j}$ converging at
least in the $C^1\loc$ topology to a limit $\widetilde{\TTT}_\infty$,
as $\tt_j\to+\infty$, we have
\begin{equation}\label{gggg}
\lim_{j\to\infty}\frac{1}{\sqrt{2\pi}}\int_{\widetilde{\TTT}_{x_0,\tt_j}}\widetilde{\rho}\,d\sigma=\frac{1}{\sqrt{2\pi}}
\int_{\TTT_\infty}\widetilde{\rho}\,d\sigma=\widehat{\Theta}(x_0)\,.
\end{equation}
\end{prop}

\begin{proof}
Assume that we have a sequence of rescaled triods
$\widetilde{\TTT}_{x_0,\tt_j}$ converging in the $C^1\loc$ topology to
a limit ${\TTT}_\infty$, as $\tt_j\to+\infty$. Since by Lemma~\ref{Elemma1}
the limit must be embedded with multiplicity one, the convergence on
every compact subset of $\R^2$ implies that the Radon measures $\HH^1\res
\widetilde{\TTT}_{x_0,\tt_{j_l}}$ weakly$^{{\displaystyle{\star}}}$--converge in $\R^2$  to the Radon measure
$\HH^1\res\TTT_\infty$. Moreover, as in the proof of Proposition~6.20
in~\cite{mannovtor}, 
we can pass to the limit in the following Gaussian integral:
$$
\lim_{j\to\infty}\frac{1}{\sqrt{2\pi}}\int_{\widetilde{\TTT}_{x_0,\tt_j}}\widetilde{\rho}\,d\sigma=\frac{1}{\sqrt{2\pi}}
\int_{\TTT_\infty}\widetilde{\rho}\,d\sigma\,.
$$
Consequently, 
$$
\frac{1}{\sqrt{2\pi}}\int_{\widetilde{\TTT}_{x_0,\tt_j}}\widetilde{\rho}\,d\sigma
=\int_{\TTT_{t(\tt_j)}}\rho_{x_0}(x,t(\tt_j))\,ds
=\Theta(x_0,t(\tt_j))\to\widehat{\Theta}(x_0)\,,
$$
as $j\to\infty$ and equality~\eqref{gggg} follows.

We now show the first statement.\\
Setting $\tt_1=-1/2\log T$ and letting $\tt_2$ go to $+\infty$ in 
the rescaled monotonicity formula~\ref{reseqmonfor-int},  by
Lemma~\ref{rescstimadib} we get
$$
\int\limits_{-1/2\log{T}}^{+\infty}\int\limits_{\widetilde{\TTT}_{x_0,\tt}}
\vert\,\widetilde{\underline{k}}+x^\perp\vert^2\widetilde{\rho}\,d\sigma\,d\tt<+\infty\,,
$$
then, {\em a fortiori},
$$
\int\limits_{{\mathcal{I}}}\int\limits_{\widetilde{\TTT}_{x_0,\tt}}
\vert\,\widetilde{\underline{k}}+x^\perp\vert^2\widetilde{\rho}\,d\sigma\,d{\tt}<+\infty\,.
$$
Being the last integral finite and being the integrand a nonnegative
function on a set of infinite Lebesgue measure, we can extract within ${\mathcal I}$ a
sequence of times $\tt_{j}\to+\infty$, such that 
$$\lim_{j\to +\infty}\int\limits_{\widetilde{\TTT}_{x_0,\tt_j}}
\vert\,\widetilde{\underline{k}}+x^\perp\vert^2\widetilde{\rho}\,d\sigma
=0\,.$$ 
It follows that, for every ball of radius $R$ in $\R^2$, the triods $\widetilde{\TTT}_{x_0,\tt_j}$ have
curvatures uniformly bounded in $L^2(B_R)$. 
Moreover, by Lemma~\ref{rescestim}, for every ball 
$B_R$ centered at the origin of $\R^2$ we have the uniform bound 
$\HH^1({\widetilde{\TTT}_{x_0,\tt_j}}\cap B_R)\leq C_R$, for some
constants $C_R$ independent of $j\in\NN$. Then, 
reparametrizing all the triods with arclength, we obtain curves with uniformly 
bounded first derivatives (from above and below away from zero by the assumption on the lengths) 
and with second derivatives in $L^2\loc$.\\
By standard compactness arguments (see~\cite{huisk3,langer2}), 
the sequence ${\widetilde{\TTT}_{x_0,\tt_{j}}}$ of reparametrized
triods admits a subsequence ${\widetilde{\TTT}_{x_0,\tt_{j_l}}}$
which converges weakly in $W^{2,2}\loc$ and also in the $C^1\loc$ topology, to a (possibly empty) set
$\TTT_{\infty}$. If the point $x_0\in\R^2$ is distinct from all the endpoints $P^i$, 
then $\TTT_{\infty}$ has no endpoints, since they go to infinity along the rescaled flow. 
If $x_0=P^i$, the set $\TTT_\infty$ has a single endpoint at the
origin of $\R^2$.\\
As we have already pointed out, by Lemma~\ref{Elemma1}, the limit set
(if not empty) has no self--intersections and multiplicity
one, moreover, if a 3--point is present then the angles are of $120$ degrees by the convergence of the
curves in $C^1\loc$.\\
Since the integral functional 
$$
\widetilde{\TTT}\mapsto
\int\limits_{\widetilde{\TTT}}\vert\,\widetilde{\underline{k}}+x^\perp\vert^2\widetilde{\rho}\,d\sigma
$$
is lower semicontinuous with respect to this convergence
  (see~\cite{simon}), the limit ${\TTT}_\infty$ distributionally satisfies 
  $\underline{k}_\infty+x^\perp=0$. 
  In principle, the limit set is composed by curves in
  $W^{2,2}\loc$, but from the relation
  ${\underline{k}}_\infty+x^\perp=0$, it follows that 
  ${\underline{k}}_\infty$ is continuous, since the curves are
  $C^1\loc$. By a bootstrap argument, it is then easy to see that 
the $\TTT_\infty$ is actually smooth.\\
  Such a limit set is an unbounded triod or curve with at most one endpoint 
  (depending on the choice of the point $x_0$), moreover, by
  Lemma~\ref{Elemma1}, (if not empty) it has no self--intersections.\\
As the relation above implies ${k}_\infty=-\langle
  x\,\vert\,\nu\rangle$ at every point $x\in{\TTT}_\infty$, repeating
  the argument of Lemma~5.2 in~\cite{mannovtor}, if a triod is
  present, it must be centered at the origin of 
  $\R^2$ and this excludes the presence of an endpoint at the same
  time. Indeed, in such case, it must be $x_0=P^1$ (for 
instance) and any blow--up must be contained in a half space (since
the triod does not ``escape'' the convex set $\Omega$ during the
evolution) which is clearly impossible for a triod.\\
Thus, by the same relation, the classification Lemmas~5.2,~5.3,~5.4
and Proposition~5.5 in~\cite{mannovtor}, we can conclude 
that in any case the curvature of the limit set is zero everywhere and
that ${\TTT}_\infty$ is among the sets in the statement.
  
Since on every ball $B_R$ the sequence of rescaled triods
${\widetilde{\TTT}_{x_0,\tt_{j}}}$ can converge (in the $C^1$
topology) only to a limit set with zero curvature, satisfying $x^\perp=0$
and 
$$
\lim_{j\to\infty}\int\limits_{\widetilde{\TTT}_{x_0,\tt_j}\cap B_R}
\vert\,\widetilde{\underline{k}}+x^\perp\vert^2\,d\sigma=0\,,
$$
as the term $x^\perp$ is continuous in the $C^1\loc$ convergence, we actually have that
$$
\lim_{j\to\infty}\int\limits_{\widetilde{\TTT}_{x_0,\tt_j}\cap B_R}
\widetilde{\underline{k}}^2\,d\sigma=0\,.
$$

Finally, the values of $\widehat{\Theta}(x_0)$ in the statement are
obtained through a computation by means of formula~\eqref{gggg}.
\end{proof}

\begin{lemma}\label{remhot}
There exists the limit $x_0=\lim_{t\to T}O(t)$, 
and corresponds to the unique point $x_0\in\overline\Omega$ such that 
$\widehat\Theta(x_0)=3/2$.\\
Moreover, the set of rescaled times
$$
{\mathcal I}_{x_0}\,=\,\bigl\{\,\tt\in[-1/2\log T, +\infty)\,\text{
  such that }\,\vert\, O(t(\tt))-x_0\vert\geq
\sqrt{2(T-t(\tt))}\,\bigr\}
$$
has finite Lebesgue measure.
\end{lemma}

\begin{proof}
We first  show that $\widehat{\Theta}(\cdot)$ can be equal to $3/2$  
only at one point in $\overline{\Omega}$. Assuming that
$\widehat{\Theta}(x_0)=\widehat{\Theta}(y_0)=3/2$, we let
\begin{equation*}
\begin{split}
{\mathcal I}_{x_0}=&\,\bigl\{\,\tt\in[-1/2\log T, +\infty)\,\text{
  such that }\,\vert\, O(t(\tt))-x_0\vert\geq
\sqrt{2(T-t(\tt))}\,\bigr\}\,,\\
{\mathcal I}_{y_0}=&\,\bigl\{\,\tt\in[-1/2\log T, +\infty)\,\text{
  such that }\,\vert\, O(t(\tt))-y_0\vert\geq
\sqrt{2(T-t(\tt))}\,\bigr\}
\end{split}
\end{equation*}
and we claim that both have finite Lebesgue measure. 
Indeed, if the Lebesgue measure of 
${\mathcal I}_{x_0}$ is not finite, we have 
$$
\int\limits_{{{\mathcal I}_{x_0}}}\int\limits_{\widetilde{\TTT}_{x_0,\tt}}
\vert\,\widetilde{\underline{k}}+x^\perp\vert^2\widetilde{\rho}\,d\sigma\,d{\tt}<+\infty\,.
$$
Hence, since we assumed $\widehat{\Theta}(x_0)=3/2$, 
 we can extract a sequence of times $\tt_{j}\in {\mathcal I}_{x_0}$
 such that the rescaled triods ${\widetilde{\TTT}_{x_0,\tt_j}}$
 converge in the $C^1\loc$ topology to 
an infinite flat triod centered at the origin of $\R^2$. This is
 clearly in contradiction with the fact that, by construction, 
every set ${\widetilde{\TTT}_{x_0,\tt_j}}\cap B_{1/2}$ does not
 contain the 3--point of the rescaled triod
 ${\widetilde{\TTT}}_{x_0,\tt_j}$.
We thus proved that ${\mathcal I}_{x_0}$ 
and ${\mathcal I}_{y_0}$ have finite measure.\\
If the points $x_0$ and $y_0$ are distinct, we have a contradiction, as $[\tt_0,+\infty)\setminus{\mathcal I}_{y_0}\subset{\mathcal I}_{x_0}$, if $\tt_0$ is large enough and 
  the set $[\tt_0, +\infty)\setminus{\mathcal I}_{y_0}$ would have
  finite Lebesgue measure as well, which is clearly not possible.

We now see that $\widehat{\Theta}(x_0)=3/2$ holds for every point
$x_0\in\overline{\Omega}$ such that there exists a sequence $t_j\to T$
with $\lim_{j\to\infty}O(t_j)=x_0$. This fact, by the compactness of
$\overline{\Omega}$ and the uniqueness of the point $x_0$, implies the
first statement of the lemma.
 
Fixing any $r\in[0,T)$,  we definitely have $t_j>r$, hence if $O(t_j)\to x_0$, we get
\begin{align*}
\Theta(x_0,r)+b(x_0,r)=&\,
    \int_{{\TT_r}} \frac{e^{-\frac{\vert
    x-x_0\vert^2}{4(T-r)}}}{\sqrt{4\pi(T-r)}}\,ds
+\int_r^T\sum_{i=1}^3\left\langle\,\frac{P^i-x_0}{2(T-t)} ,\tau^i(1,t)\right\rangle
\frac{e^{-\frac{\vert
    P^i-x_0\vert^2}{4(T-t)}}}{\sqrt{4\pi(T-t)}}\,dt\\
=&\,\lim_{j\to\infty}\Biggl\{
 \int_{{\TT_r}} \frac{e^{-\frac{\vert
    x-O(t_j)\vert^2}{4(t_j-r)}}}{\sqrt{4\pi(t_j-r)}}\,ds
+\int_r^{t_j}\sum_{i=1}^3\left\langle\,\frac{P^i-O(t_j)}{2(t_j-t)} ,\tau^i(1,t)\right\rangle
\frac{e^{-\frac{\vert
    P^i-O(t_j)\vert^2}{4(t_j-t)}}}{\sqrt{4\pi(t_j-t)}}\,dt\Biggr\}\\
\geq&\,\lim_{j\to\infty}\lim_{r\to t_j^-}\Biggl\{
 \int_{{\TT_r}} \frac{e^{-\frac{\vert
    x-O(t_j)\vert^2}{4(t_j-r)}}}{\sqrt{4\pi(t_j-r)}}\,ds
+\int_r^{t_j}\sum_{i=1}^3\left\langle\,\frac{P^i-O(t_j)}{2(t_j-t)} ,\tau^i(1,t)\right\rangle
\frac{e^{-\frac{\vert
    P^i-O(t_j)\vert^2}{4(t_j-t)}}}{\sqrt{4\pi(t_j-t)}}\,dt\Biggr\}\\
=&\,\lim_{j\to\infty}\lim_{r\to t_j^-}
\int_{{\TT_r}} \frac{e^{-\frac{\vert
    x-O(t_j)\vert^2}{4(t_j-r)}}}{\sqrt{4\pi(t_j-r)}}\,ds\,,
    \end{align*}
where the inequality follows from Proposition~\ref{promono} and in the last passage we applied 
Lemma~\ref{stimadib} with $t_j$ in place of $T$. 
Indeed, the monotonicity formula (and actually all the previous strategy) 
holds also if $T$ is not the maximal existence time. 
Repeating all the argument in the Proposition~\ref{resclimit} at time $t_j$, we then see that 
  the last integral inside the limit must be equal to $3/2$ (as we are
  rescaling exactly around the 3--point $O(t_j)$) and thus the only
  possible limit of rescaled triods is an unbounded  
  triod in $\R^2$ centered at the origin.\\
  Hence, we can conclude that for every $r\in[0,T)$ there holds
  $\Theta(x_0,r)+b(x_0,r)\geq 3/2$, which, when $r\to T$, implies that
  $\widehat{\Theta}(x_0)=3/2$.
\end{proof}

In the following, given $\overline{x}\in\R^2$ and $R>0$, we denote by $Q_R(\overline{x})$ the square
$$Q_R(\overline{x}):=\left\{ x\in\R^2:\  |x_1-\overline{x}_1|\le R,\, |x_2-\overline{x}_2|\le R\right\}.$$

\begin{prop}\label{prograph}
Suppose that the curve $\gamma_0$ is a graph over $\langle e_1\rangle$
in the square $Q_{2R}(x_0)$, and assume that the curve 
$\gamma_t\cap Q_{2R}(x_0)$ is contained in the horizontal strip
$\{|x_2|\le \delta\}$ for any $t\in[0,\tau)$, with $\tau>0$ and 
$0<\delta<R$.\\
Then $\gamma_t\cap Q_{2R}(x_0)$ is a graph over $\langle e_1\rangle$ for all $t\in[0,\tau)$.
\end{prop}

\begin{proof}
We claim that the number of intersections of $\gamma_t$ with any vertical segment of the form
$\ell_x:=\{x+s e_2: x\in Q_{2R}(x_0), s\in\R\}\cap Q_{2R}(x_0)$
is nonincreasing in time, hence it is constantly equal to 1 as $\gamma_0$ is a graph in $Q_{2R}(x_0)$ over $\langle e_1\rangle$.
It then follows that $\gamma_t\cap Q_{2R}(x_0)$ is a graph over
$\langle e_1\rangle$ for all $t\in[0,\tau)$ and the thesis is proven.

In order to prove the claim, let us assume by contradiction that there
exist a vertical segment $\ell$ and a time ${\overline{t}}\ge 0$ such
that 
the set $\gamma_t\cap \ell$ is a single point for $t\in
[0,{\overline{t}})$, and has cardinality strictly greater than 1 for a
sequence 
$t_n\downarrow {\overline{t}}$. In particular, there exist a point
${\overline{x}}\subset \gamma_{{\overline{t}}}\cap \ell$ and two sequences
$x_n,\,y_n$ such that
\begin{equation}\label{xyn}
\textrm{
$x_n,y_n\in \gamma_{t_n}\cap \ell$, $x_n\ne y_n$
and $\lim_{n\to\infty} x_n = \lim_{n\to\infty} y_n={\overline{x}}$.}
\end{equation}
It follows that $\gamma_{{\overline{t}}}$ has a vertical tangent line
at ${\overline{x}}$, so that we can write $\gamma_{{\overline{t}}}\cap
Q_\delta({\overline{x}})$ as 
a smooth graph over $\ell$ for a suitably small $\delta>0$.

By~\cite[Theorem 2.1]{eckhui2} there exists $\eps>0$ such that
$\gamma_t\cap Q_\frac\delta 2({\overline{x}})$ is also a graph over
$\ell$ for all 
$t\in [{\overline{t}}, {\overline{t}}+\eps]$ and, by continuity, at
the intersection with $\partial Q_{\frac{\delta}{2}}(\overline{x})$
the curve $\gamma_t$ does not intersect $\ell$. 
Then, by the Sturmian theorem of Angenent in~\cite[Proposition~1.2]{angen2} and~\cite[Section~2]{angen1}
(see~\cite{angen5} for the proof), we have that the cardinality of
$\gamma_t\cap\ell$ in $Q_\frac\delta 2({\overline{x}})$ is 
nonincreasing in time on $[{\overline{t}}, {\overline{t}}+\eps]$, thus
contradicting property~\eqref{xyn}.
\end{proof}

\begin{cor}\label{corgraph}
Assume that $\g_0\cap B_{7R}(x_0)$ is a graph over $\langle e_1\rangle$, contained in the
horizontal strip $\{|x_2|\le R\}$.
Then $\g_t\cap B_{2R}(x_0)$ is a graph over $\langle e_1\rangle$
for all $t\in [0,\tau)$, with $\tau = R^2/2$.
Moreover, letting $v=\langle\nu\,\vert\,e_2\rangle^{-1}$, we have
$$
\sup_{t\in [0,\tau)}\sup_{\gamma_t\cap B_{R}\,(x_0)}v\leq C\sup_{\gamma_0\cap B_{2R}\,(x_0)}v
$$
for some $C>0$ independent of $R$.
\end{cor}

\begin{proof}
Letting $x_\pm = x_0\pm 4Re_2$, by assumption we have that $\g_0$ is
contained in the complementary of 
the set $B_{3R}(x_+)\cup B_{3R}(x_-)\subset B_{7R}(x_0)$.

By comparison principle, it follows that
$\g_t$ does not intersect the set $B_{R(t)}(x_+)\cup B_{R(t)}(x_-)$,
with $R(t)=\sqrt{9R^2-2t}$, for all $t\in [0,9R^2/2)$. 
In particular, $\g_t\cap Q_{2R}(x_0)$ does not intersect the upper and
lower edge of the square $Q_{2R}(x_0)$ if $t\in [0,\tau)$, with
$\tau=R^2/2$. 
Therefore, from Proposition~\ref{prograph} it follows that $\g_t\cap
Q_{2R}(x_0)$, hence also $\g_t\cap B_{2R}(x_0)$, 
is a graph over $\langle e_1\rangle$ for all $t\in [0,R^2/2)$.\\
The last assertion of the corollary then follows from Theorem~2.3 in~\cite{eckhui2},
noticing that if $\g_t$ is the graph of the function $u(\cdot\,, t)$, then 
$v = \sqrt{1+\vert u^\prime\vert^2}$.
\end{proof}

We recall the following result~\cite[Corollary~3.2 and Corollary~3.5]{eckhui2}.

\begin{prop}\label{teoeh}
Suppose that
$\gamma_t$ is a graph over $\langle e_1\rangle$ in $B_{R}(x_0)$
for all $t\in[0,\tau)$. Then
letting $\theta\in (0,1)$ and $m\geq0$, we have
\begin{equation*}
\sup_{\gamma_t\cap B_{\sqrt{\theta R^2-2t}}\,(x_0)}t^{m+1}|\partial_s^m k|^{2}\leq
C_{m,v}
\end{equation*}
for all $t\in[0,\tau)$, where the constant $C_{m,v}$ depends only on $m$, $\theta$ and
$\sup_{t\in[0,\tau)}\sup_{\gamma_t\cap B_{\sqrt{R^2-2t}}\,(x_0)} v$.
\end{prop}

\begin{prop}\label{prokappa}
Let $\gamma_t$ be as in Proposition~\ref{teoeh}.
For all $\theta\in (0,1)$ we have
\begin{equation}\label{iefin}
\sup_{\gamma_t\cap B_{\sqrt{\theta R^2-2t}}\,(x_0)}|k|^{2}\leq \frac{C_v}{(1-\theta)^2}\biggl( \frac{1}{R^2}
+\sup_{\gamma_0\cap B_{R}\,(x_0)}|k|^{2}\biggr)
\end{equation}
for all $t\in[0,\tau)$, where the constant $C_v$ depends only on $\sup_{t\in[0,\tau)}\sup_{\gamma_t\cap B_{R}\,(x_0)} v$.
\end{prop}
\begin{proof}
Let $g=k^2\p(v^2)$, with $\p(s):=s(1-cs)^{-1}$ and $c>0$ to be chosen later,
and let $\eta=(R^2-|x|^2-2t)^2$. By a direct computation as in the proof of Theorem~3.1 in~\cite{eckhui2}, we obtain
\begin{equation}\label{eqprima}
\left( \partial_t-\Delta\right)g\eta \le
-2cg^2\eta 
-2\bigl\langle \p v^{-3}\nabla v + \eta^{-1}\nabla\eta,\nabla(g\eta)\bigr\rangle
+ C(n)\left( \Bigl( 1+\frac{1}{cv^2}\Bigr)\Bigl(|x|^2+2t\Bigr)+R^2\right)g\,.
\end{equation}
At a point where $m(t):=\max_{\gamma_t\cap B_{\sqrt{R^2-2t}}(x_0)} \, (g\eta)$ is attained in space,
multiplying inequality~\eqref{eqprima} by $\frac{\eta}{2c}$ we get
\[
m^\prime(t) \frac{\eta}{2c} \le - m(t)^2 + \frac{C(n)}{2c}\left( \Bigl( 1+\frac{1}{cv^2}\Bigr)\Bigl(|x|^2+2t\Bigr)+R^2\right)m(t)\,,
\]
which yields $m^\prime(t)\le 0$ as soon as
\[
m(t)\ge \frac{C(n)}{2c}\left( \Bigl( 1+\frac{1}{cv^2}\Bigr)\Bigl(|x|^2+2t\Bigr)+R^2\right)\,.
\]
Choosing 
$$c=\frac{1}{2}\inf_{t\in [0,\tau)}
\inf_{\gamma_t\cap B_{\sqrt{R^2-2t}}(x_0)} \, v^{-2},
$$ 
we obtain 
$$
cv^2\leq \frac{1}{2} 
\left(\inf_{\gamma_t\cap B_{\sqrt{R^2-2t}}(x_0)} \, v^{-2}\right)
\left(\sup_{\gamma_t\cap B_{\sqrt{R^2-2t}}(x_0)} \, v^{2}\right)=1/2
$$
in $\gamma_t\cap B_{\sqrt{R^2-2t}}(x_0)$, hence 
$\p(v^2)\le 2v^2$, and the function $\eta g$ is well defined.
Moreover we have $m^\prime(t)\le 0$ whenever
\[
m(t)\ge 4C(n)\, R^2 \sup_{\gamma_t\cap B_{\sqrt{R^2-2t}}(x_0)} \, v^2.
\]
As $\p(v)\ge 1$ and $\eta\ge (1-\theta)^2R^4$ in $B_{\sqrt{\theta R^2-2t}}\,(x_0)$, it follows
\begin{eqnarray*}
\sup_{\gamma_t\cap B_{\sqrt{\theta R^2-2t}}\,(x_0)}|k|^{2}&\le&
(1-\theta)^{-2}R^{-4}\sup_{\gamma_t\cap B_{\sqrt{R^2-2t}}(x_0)} \,
(g\eta)\\
&\le& (1-\theta)^{-2}
\max\,\biggl\{\frac{m(0)}{R^{4}}, \frac{4C(n)}{R^{2}}
\,\sup_{\gamma_t\cap B_{\sqrt{R^2-2t}}(x_0)} \, v^2\biggr\}
\end{eqnarray*}
which gives estimate~\eqref{iefin} as 
$$
m(0)\leq 2 R^4 \left(\sup_{\gamma_0\cap B_{R}(x_0)} \, v^2\right)
\left(\sup_{\gamma_0\cap B_{R}(x_0)}\, |k|^2\right).
$$
\end{proof}

\section{Proof of Theorem~\ref{corfin}}

The fact that if $T=+\infty$ and the lengths of the three curves of a
triod moving by curvature are bounded below away from zero uniformly
in time, then the evolving triod $\TT_t$
tends as $t\to +\infty$ to the unique Steiner triod connecting the
three fixed endpoints is shown in~Section 8 of~\cite{mannovtor}.

This section is devote to exclude finite time singularities
(i.e. $T<+\infty$) for a triod moving by curvature, whose curves have lengths bounded away from zero from below,
uniformly in time. From this fact, Theorem~\ref{corfin} follows.

To this aim, we will proceed with an argument by contradiction relying on the
$C^1\loc$ convergence (with the $L^2$ norm of the curvature going to
zero in every compact subset of $\R^2$) of a sequence of rescaled
triods to any of the three singularity models in
Proposition~\ref{resclimit}. The argument is similar in spirit to
the one in~\cite{manmaggra}, adapted to the case of an evolving triod.

To set the notation, let $F: \TT\times [0,T) \to \R^2$, with $T <
\infty$, be a triod moving 
by curvature in its maximal time interval of smooth
existence. We assume that the lengths of the three curves of the triod
$\TT_t$ are uniformly bounded below away from zero and that $T<+\infty$. We
are going to show that the full $L^2$ norm of the 
curvature of the evolving triod stays uniformly bounded up to time $T$,
hence contradicting Proposition~\ref{curvexplod}.

We define the set of {\em reachable points} of the flow as
$$
\mathcal{R} = \bigl\{ x \in \R^2\,\bigl\vert\,\text{ there exist $p_i
\in \TT$ and $t_i \nearrow T$ such that $\lim_{i \to \infty}F(p_i , t_i) = x$}\bigr\}\,.
$$
Such a set is easily seen to be closed, contained in
$\overline{\Omega}$, hence compact, and the following lemma holds.

\begin{lemma}
A point $x \in \R^2$ belongs to $\mathcal{R}$ if and only if for every
time $t \in [0,T)$ the closed ball with center $x$ and radius
$\sqrt{2(T-t)}$ intersects $\TT_t$. 
\end{lemma}

\begin{proof}
One of the two implications is trivial. We have to prove that if $x
\in \mathcal{R}$, then $F(\TT, t) \cap \overline{B}_{\sqrt{2(T-t)}}(x)
\neq \emptyset$. If $x$ is one of
the endpoints, the result is obvious, otherwise we define the function $d_x(t) = \inf_{p \in \TT} |F(p,t) - x|$, 
where, due to the compactness of $\TT$ the infimum is actually a
minimum and definitely, as $t\to T$, let us say for $t>t_x$ it cannot be taken at an endpoint,
by the assumption $x\in\mathcal{R}$. Since the function $d_x: [0,T) \to \R$ is locally Lipschitz, 
we can use Hamilton's trick to compute its time derivative and get 
(for any point $q$, different by an endpoint, where at time $t$ the minimum of $|F(p,t)
- q|$ is attained)
\begin{equation*}
\begin{split}
\partial_t d_x(t) & = \partial_t |F(q,t) - x| \geq \frac{\op k(q,t)\nu(q,t) + \l(q,t)
  \t(q,t) , F(q,t) - x \cl}{|F(q,t) - x|} \\
& = \frac{\op k(q,t) \nu(q,t)  , F(q,t)- x \cl}{|F(q,t) - x|} \geq -\frac{1}{d_x(t)}\,,
\end{split}
\end{equation*}
since at a point of minimum distance the vector 
$ \frac{F(q,t) -  x}{|F(q,t) - x|} $ 
is parallel to $\nu(q,t)$. Integrating
this inequality over time, we get
$$
d^2_x(t) - d^2_x(s) \leq 2(s-t) 
\qquad \quad{\rm for\ }s>t>t_x\,.
$$
We now use the hypothesis that $x$ is reachable (i.e. $\lim_{t_i \to T}
d_x(t_i) = 0$) and we conclude
$$
d^2_x(t) = \lim_{i \to \infty} [d^2_x(t) - d^2_x(t_i)] \leq 2 \lim_{i
  \to \infty} (t_i - t) = 2 (T-t)\,,
$$
for every $t>t_x$.
\end{proof}

As a consequence, when we consider the blow--up of the evolving
triods around points of $\overline{\Omega }$, we have a dichotomy
among them. Either the limit of any sequence of rescaled triods is
not empty and we are rescaling around a point in $\mathcal{R}$, or the
blow--up limit is empty, since the distance of the evolving triod from
the point of blow--up is too big. Conversely, if the blow up
point belongs to $\mathcal{R}$, the above lemma ensures that any
rescaled triod contains at least one point of the closed unit ball of
$\R^2$.

Fixing any point $x_0\in\mathcal{R}$, by 
Proposition~\ref{resclimit} there is a sequence $\tt_{i}
\nearrow \infty$ of rescaled triods such that
$\widetilde{\TT}_{x_0,\tt_j}$ converges in the $C^1\loc$ topology to a
nonempty limit which must be either a straight line, a halfline or an
infinite flat triod. Moreover, in every ball $B_R\in\R^2$, the $L^2$ norm of the curvature  along such sequence goes to zero as $j\to\infty$.

We start considering the case when the blow--up limit is a straight line.

\begin{prop}\label{noline}
If the sequence of rescaled triods $\widetilde{\TT}_{x_0,\tt_j}$ converges to a
straight line, then the curvature of the evolving triod is uniformly
bounded for $t\in[0,T)$ in a ball around the point $x_0$.
\end{prop}

\begin{proof}
Assume that there is a straight line $L$ through the
origin of $\R^2$ such that the sequence of rescaled triods
$\widetilde{\TT}_{x_0,\tt_j}$ converges to $L$ as $j\to\infty$.\\
Recalling Lemma~\ref{remhot} this implies that the distance
$|O(t)-x_0|$ is uniformly bounded from below, 
so that there exists $i\in \{1,2,3\}$ such that the rescaled curves 
$\frac{\g^i_{t_{j}}}{\sqrt{2(T-t_{j})}}$ converge to $L$ as $j\to\infty$. In particular,
for all $M>1$ there exists $j_M\in\mathbb N$ such that the curve 
$\g^i_{t_{j_M}} \cap B_{7M\sqrt{2(T-t_{j_M})}}(x_0)$ is a graph over the line $x_0+L$.
By Corollary~\ref{corgraph} it follows that 
$\g^i_{t} \cap B_{M\sqrt{2(T-t_{j_M})}}(x_0)$ is also a graph over the line $x_0+L$ 
for all $t\in [t_{j_M}, t_{j_M}+M^2(T-t_{j_M}))\supset [t_{j_M}, T)$,
and its slope $v^i$ (with respect to the line $x_0+L$) is 
uniformly bounded by a constant independent of $M$ and $t$.
Therefore, if $M>2$, from Proposition~\ref{teoeh} (applied with
$\theta=1/2$) it follows that the curvature of the curve $\g^i_{t}
\cap B_{M\sqrt{2(T-t_{j_M})}}(x_0)$ and all its derivatives are
bounded for $t\in [t_{j_M},T)$ and we are done.
\end{proof}

We then consider the case of a halfline.

\begin{prop}\label{nohalf}
If the sequence of rescaled triods $\widetilde{\TT}_{x_0,\tt_j}$
converges to a halfline, then the curvature of the evolving triod is
uniformly bounded for $t\in[0,T)$ in a ball around the point $x_0$.
\end{prop}

\begin{proof}
By the $C^1\loc$ convergence of the rescaled flow to the halfline,
we can see that the point $x_0$ must be one of the
endpoints of the triod, which we will denote with $P$. We now
perform a reflexion with center $P$ of the triod and we consider the
motion by curvature of the union of the two (mutually reflected
through $P$) triods which is still a motion by curvature, now of a
{\em network} of curves (see~\cite{mannovtor} for more details). 
Since at the endpoint $P$ the curvature vanishes by construction, 
the point $P$ stays fixed during the motion of the network and the
sequence of rescaled networks around $P=x_0$ converges in the $C^1\loc$
topology to a straight line. We can now repeat the proof of
Proposition~\ref{noline} to conclude.
\end{proof}

If there is no $x_0 \in \R^2$ with $\widehat{\Theta}(x_0)=3/2$,
by Propositions~\ref{noline} and~\ref{nohalf}, there exists a ball around every reachable 
point in which the curvature of the evolving triod is uniformly bounded for $t\in[0,T)$.\\
As the set of reachable points ${\mathcal{R}}$ is compact, it follows
that the curvature is uniformly bounded as $t\to T<+\infty$, which is
contradiction to Proposition~\ref{curvexplod}. Hence, we can assume that at some (unique) point $x_0\in\Omega$ we have $\widehat{\Theta}(x_0)=3/2$ and that the sequence of rescaled triods
$\widetilde{\TT}_{x_0,\tt_j}$ converges to 
an infinite flat triod $\TT_\infty$ centered at the origin. Furthermore,
the $L^2$ norm of the curvature of the rescaled triods goes to zero on
every compact subset of $\R^2$. 
By Lemma~\ref{remhot} this means that $x_0$ is the limit
of the 3--point $O(t)$ as $t\to T$. 
We write $\TT_\infty=L^1\cup L^2\cup L^3$ where the $L^i$'s are
halflines from the origin of $\R^2$.

In order to analyze the case of a flat triod arising as a blow--up
limit, we need some preliminary estimates, based on the following
Gagliardo--Nirenberg interpolation inequalities 
(see~\cite{adams,aubin0}, for instance).
\begin{prop}[Proposition~3.11 in~\cite{mannovtor}] 
\label{propint}
Let $\gamma$ be a smooth regular curve in
  $\R^2$ with finite length ${\mathrm L}$. If $u$ is a smooth function defined on
  $\gamma$ and $m\geq1$, $p\in[2,+\infty]$, we have the estimates
  \begin{equation*}
    {\Vert\partial_s^n u\Vert}_{L^p}
    \leq C_{n,m,p}
      {\Vert\partial_s^m  u\Vert}_{L^2}^{\sigma}
      {\Vert u\Vert}_{L^2}^{1-\sigma}+
      \frac{B_{n,m,p}}{{\mathrm L}^{m\sigma}}{\Vert u\Vert}_{L^2}
  \end{equation*}
  for every $n\in\{0,\dots, m-1\}$ where
$$
\sigma=\frac{n+1/2-1/p}{m}
$$
and the constants $C_{n,m,p}$ and $B_{n,m,p}$ are independent of $\gamma$.
\end{prop}

\begin{lemma}\label{kkevol}
Let $F: \TT\times [0,T) \to \R^2$, with $T < \infty$, be a triod moving
by curvature {\em with moving endpoints} $Q^i:[0,T)\to\Omega$ such that
the lengths of the three curves are uniformly bounded from below away from
zero by $L>0$.\\
Then, for some constants $C_1>0$, $C_2>0$, independent of the triod, the following
estimate holds:
\begin{equation*}
\frac{d}{dt} \int_{\TT_t} k^2 \,ds \leq C_1 \Big( \int_{\TT_t} k^2 \,ds
\Big)^3 + \frac{C_2}{L}\Big(\int_{\TT_t} k^2\, ds\Big)^2 + 2 \sum^3_{i=1}\, k^i (k_s^i+\lambda^ik^i)\,
\biggr\vert_{\text{{\em at the point} $Q^i(t)$}} \,.
\end{equation*}
\end{lemma}

\begin{proof}
Using Lemma~\ref{evoluzioni} and integrating by parts (for more details
refer to computations~(3.4) and~(3.5) in~\cite{mannovtor}), 
we get
\begin{align*}
\frac{d}{dt} \int_{\TT_t} k^2\, ds =&\, - 2 \int_{\TT_t} k_s^2\, ds 
+ \int_{\TT_t} k^4\, ds - \sum^3_{i=1} \,k^i(k^i_s+\l^i
{k^i})\,\biggr\vert_{\text{{ at      the 3--point}}}\\
&\,+ 2 \sum^3_{i=1}\, k^i (k_s^i+\lambda^ik^i)\,
\biggr\vert_{\text{{ at the point} $Q^i(t)$}} \\
=&\, - 2 \int_{\TT_t} k_s^2\, ds 
+ \int_{\TT_t} k^4\, ds + 2 \sum^3_{i=1}\, k^i (k_s^i+\lambda^ik^i)\,
\biggr\vert_{\text{{ at the point} $Q^i(t)$}}\,,
\end{align*}
where we applied the ``orthogonality'' relation~(2.10)
in~\cite{mannovtor}, saying that the 3--point contribution above is zero.\\
Letting $L$ to be the minimum of the length of the three curves of the triod, 
by Proposition~\ref{propint} (applied to $u=k$
and having set $p=4,\,n=0,\,m=1,\,\sigma=1/4$) 
and Peter--Paul
inequality, for any $\e > 0$ we have the interpolation estimate
\begin{equation*}
\begin{split}
\int_{\TT_t} k^4\,ds & \leq\left[ C \Big( \int_{\TT_t} k_s^2\, ds \Big)^{1/8}
\Big( \int_{\TT_t}\, k^2 ds \Big)^{3/8} +\frac{C}{L^{1/4}}\Big( \int_{\TT_t} k^2\, ds\Big)^{1/2}\right]^4 \\
& \leq C \Big( \int_{\TT_t} k_s^2\, ds \Big)^{1/2}
\Big( \int_{\TT_t}\, k^2 ds \Big)^{3/2} + \frac{C}{L} \Big( \int_{\TT_t} k^2\, ds\Big)^2 \\
& \leq \e \int_{\TT_t} k_s^2\, ds + C_1\Big(
\int_{\TT_t}k^2\, ds \Big)^3+ \frac{C_2}{L}\Big(\int_{\TT_t} k^2\, ds\Big)^2\,,\\
\end{split}
\end{equation*}
where the constants $C_1,\,C_2$ depend on $\eps$.
Substituting in the last equation above, after taking $\e<2$, we get the thesis.
\end{proof}

We are now ready to prove the main result of the paper.

\begin{proof}[Proof of Theorem~\ref{corfin}] 
Since the subset $\mathcal I$ of  $[-1/2\log T,+\infty)$ defined by ${\mathcal{I}}=\cup_{j=1}^\infty(\tt_j+\log{\sqrt{3/2}},\tt_j+\log{\sqrt{3}})$ has obviously 
infinite Lebesgue measure, by Proposition~\ref{resclimit}, we can
assume that there exists another sequence of rescaled triods
$\widetilde{\TT}_{x_0,\widetilde{\tt}_j}$, with
$\widetilde{\tt}_j\in(\tt_j+\log{\sqrt{3/2}},\tt_j+\log{\sqrt{3}})$
for every $j\in\NN$,  which is also $C^1\loc$ converging to a flat
triod ({\em a priori} not necessarily the same one) centered at the origin
of $\R^2$ as $j\to\infty$. Indeed, even if 
the two blow--up limits are different, they both must be a flat triod,
as equality~\eqref{gggg} must hold for both of them. Moreover, the $L^2$ norm of the curvature
of the modified sequence of rescaled triods, as well as the one of the original sequence of rescaled triods, converges to zero on every compact subset of $\R^2$.\\
Finally, passing to a subsequence, we can also assume that $\tt_j$ and
$\widetilde{\tt}_j$ (hence, also $t_j$ and $\widetilde{t}_j$) are
increasing sequences.\\
Notice that, by means of the rescaling relation
$\tt(t)=-\frac{1}{2}\log{(T-t)}$, the condition
$\widetilde{\tt}_j\in(\tt_j+\log{\sqrt{3/2}},\tt_j+\log{\sqrt{3}})$
reads, for the original time parameter, as
$\widetilde{t}_j\in\left(\frac23 t_{j_M}+ \frac 13 T,\frac 13 t_{j_M}
  + \frac23 T\right)$.

Repeating the argument in the proof of Proposition~\ref{noline}, 
for any $M$ large enough there exists $j_M$ such that for all
$i\in\{1,2,3\}$ the curve $\g^i_{t}\cap B_{5M\sqrt{2(T-t_{j_M})}}(x_0)\setminus B_{M\sqrt{2(T-t_{j_M})}}(x_0)$ 
is a graph over $x_0+L^i$ for all $t\in [t_{j_M},T)$, with slope (with
respect to the line $x_0+L^i$) uniformly bounded by a constant
$C_v$ independent of $M$ and $t\in[t_{j_M},T)$ 
(here and in the sequel we denote by $C_v$ 
a generic constant, depending on $v$, which may vary from line to line). Moreover, by
Lemma~\ref{remhot}, we can also assume that the 3--point $O(t)$ in
this time interval does not get into the annulus 
$B_{5M\sqrt{2(T-t_{j_M})}}(x_0)\setminus
B_{M\sqrt{2(T-t_{j_M})}}(x_0)$.\\
By Proposition~\ref{teoeh}, with $\theta<1/2<9/16+\frac{1}{2M^2}$, it follows that the subsequent evolution of the curves 
$$
\gamma^i_{t_M}\cap \Bigl(B_{4M\sqrt{2(T-t_{j_M})}}(x_0)\setminus B_{2M\sqrt{2(T-t_{j_M})}}(x_0)\Bigr)\,,
$$
that, with an abuse of notation as we cannot exclude that other parts of
$\TT_t$ get into the annulus $B_{4M\sqrt{2(T-t_{j_M})}}(x_0)\setminus
B_{2M\sqrt{2(T-t_{j_M})}}(x_0)$, we still denote by 
$$
\gamma^i_{t}\cap \Bigl(B_{4M\sqrt{2(T-t_{j_M})}}(x_0)\setminus B_{2M\sqrt{2(T-t_{j_M})}}(x_0)\Bigr)\,,
$$
for $i\in\{1,2,3\}$, are smooth evolutions for all $t\in [t_{j_M},T)$ and 
the following estimate holds
\begin{equation}
\label{stimaks}
|k_s^i(t)|^2 \le\frac{C_v}{(t-{t}_{j_M})^2}\le\frac{C_v}{(\widetilde{t}_{j_M}-{t}_{j_M})^2}\leq 
\frac{9C_v}{(T-{t}_{j_M})^2}\,,
\end{equation}
for all $t\in [\widetilde t_{j_M},T)$, where the constant $C_v$
depends only on the slope with respect to the line $x_0+L^i$.\\
Since, by Proposition~\ref{resclimit}, the
$L^2$ norm of the curvature (in the rescaled ball
$\widetilde{B}_{5M}(0)$) of the 
sequence of rescaled triods $\widetilde{\TT}_{x_0,\widetilde{\tt}_j}$, which is given by 
$$
\sqrt{2(T-\widetilde{t}_j)}\int_{\TT_{\widetilde{t}_j}\cap B_{5M\sqrt{2(T- t_{j_M})}}(x_0)} k^2\, ds\,,
$$
converges to zero as $j\to\infty$, the above estimate~\eqref{stimaks} on the
derivative of the curvature, which for the sequence of rescaled triods
becomes $\vert\widetilde{k}_s^i(\tt_j)\vert\leq 3\sqrt{C}$, implies
that the $L^\infty$ norm of the curvature of the rescalings of the curves 
$$
{\g}^i_{\widetilde{t}_j} \cap \Bigl(B_{4M\sqrt{2(T-t_{j_M})}}(x_0)\setminus B_{2M\sqrt{2(T-t_{j_M})}}(x_0)\Bigr)\,,
$$
which is given by 
$$
\sqrt{2(T-\widetilde{t}_j)}\,\left(\sup_{\TT_{\widetilde{t}_j}\cap
    \bigl(B_{4M\sqrt{2(T- t_{j_M})}}(x_0)\setminus B_{2M\sqrt{2(T-
        t_{j_M})}}(x_0)\Big)} \vert k\vert\right)\,,
$$
converges to zero as $j\to\infty$.\\
Since the above argument holds not only for $j_M$ but for every $j\geq
j_M$, fixed any $\eps\in (0,1/2)$, first considering an $M>2$ large enough and then choosing a suitably large $j_M$, we can
assume that
\begin{align}
\bullet\,\,\,\,&\,\text{$M>\max\{1/\sqrt{\eps},C_2/\eps^{1/3}\}$, where the
  constant $C_2$ is the one appearing in
  Lemma~\ref{kkevol},}\nonumber\\
\bullet\,\,\,\,&\,\int_{\TT_{\widetilde{t}_{j_M}}\cap B_{5M\sqrt{2(T- t_{j_M})}}(x_0)}
k^2\, ds \le\frac{\eps}{\sqrt{2(T-\widetilde{t}_{j_M})}}\le\frac{\sqrt{3}\eps}{\sqrt{2(T-{t}_{j_M})}}\,,\label{eqeps1}\\
\bullet\,\,\,\,&\,\sup_{\TT_{\widetilde{t}_{j_M}}\cap\bigl(B_{4M\sqrt{2(T-
      t_{j_M})}}(x_0)\setminus B_{2M\sqrt{2(T-t_{j_M})}}(x_0)\bigr)}
k^2 \le\frac{\eps}{{2(T-\widetilde{t}_{j_M})}}
\le\frac{3\eps}{{2(T-{t}_{j_M})}}\,.\label{eqeps2}
\end{align}

By Proposition~\ref{prokappa}, as $M>2$, at the points 
$$
\g^i_{t} \cap \Bigl(B_{\frac{7}{2}M\sqrt{2(T-t_{j_M})}}(x_0)\setminus B_{\frac{5}{2}M\sqrt{2(T-t_{j_M})}}(x_0)\Bigr)\,,
$$
we have the estimate
\begin{equation*}
|k^i(t)|^2 \le C_v \left( \sup_{\g^i_{\widetilde t_{j_M}} 
\cap \big(B_{4M\sqrt{2(T-t_{j_M})}}(x_0)\setminus
B_{2M\sqrt{2(T-t_{j_M})}}(x_0)\big)}|k^i|^{2}
+ \frac{1}{M^2(T-t_{j_M})}\, \right)
\end{equation*}
for all $t\in [\widetilde t_{j_M},T)$, with a constant $C_v$ 
depending only on the slope of the curve with respect to the line
$x_0+L^i$, which is uniformly bounded. Thus, by the above estimate~\eqref{eqeps2} we
get
\begin{equation}\label{equanime}
|k^i(t)|^2 \le \frac{C_v}{T-t_{j_M}} \Bigl({\eps} + \frac{1}{M^2}\,\Big)\leq
 \frac{2C_v\eps}{T-t_{j_M}}\,
\end{equation}
as we already chose $M^2>1/{\eps}$ above, 
for all the points of the curve $\g^i_{t} \cap
\Bigl(B_{\frac{7}{2}M\sqrt{2(T-t_{j_M})}}(x_0)\setminus
B_{\frac{5}{2}M\sqrt{2(T-t_{j_M})}}(x_0)\Bigr)$ and times $t\in
[\widetilde t_{j_M},T)$. We want to underline once more that the constant $C$ depends only on the slope
of the curve with respect to the line $x_0+L^i$.

It follows that for every $t\in [\widetilde{t}_{j_M},T)$, all the
triods $\widehat{\TTT}_t$ determined by ``cutting'' $\TT_t$ at the new
(moving in time) endpoints $Q^i(t)=\g^i_{t} \cap\partial
B_{3M\sqrt{2(T-t_{j_M})}}(x_0)$ have the lengths of their three curves
uniformly bounded away from zero from below and unit tangent vectors at the
endpoints $Q^i(t)$ which form angles with the respective velocity
vectors $\partial_tQ^i(t)$ which are also bounded away from zero,
uniformly in time, because of the uniform control on the slope of the
curves with respect to the line $x_0+L^i$. This implies that the norm
of the curvature $\vert k^i(Q^i(t))\vert$ at any endpoint $Q^i(t)$
controls the norm of the tangential velocity
$\vert\lambda^i(Q^i(t))\vert$, up to a multiplicative constant $C_v$ (depending only on the slope),
uniformly bounded in time for $t\in [\widetilde t_{j_M},T)$.\\
Then, from estimates~\eqref{stimaks},~\eqref{equanime}, we conclude
\begin{align*}
\Bigl|k^i(Q^i(t))k_s^i(Q^i(t))\Bigr|\le&\, \frac{C_v{\eps}^{1/2}}{(T-t_{j_M})^\frac 3 2}\,,\\
\Bigl|[k^i(Q^i(t))]^2\lambda^i(Q^i(t))\Bigr|\le&\, C_v\Bigl\vert
k^i(Q^i(t))\Bigr\vert^3 \le\frac{C_v{\eps}^{3/2}}{(T-t_{j_M})^\frac 3
  2}\,,
\end{align*}
for every $t\in [\widetilde t_{j_M},T)$, where the constant $C_v$ depends only on the slope
of the curve with respect to the line $x_0+L^i$. Moreover, we can clearly always increase $j_M$ as
we like without affecting $C_v$ (this is actually true for every constant depends only on the slope
of the curve with respect to the line $x_0+L^i$), since, by Proposition~\ref{noline}, as $j\to\infty$, the three
curves
$$
\gamma_{t_j}^i\cap B_{5M\sqrt{2(T- t_{j_M})}}(x_0)\setminus
B_{M\sqrt{2(T-t_{j_M})}}(x_0)
$$
converge to a smooth limit. Hence, we can also assume that
$2C_v\eps^{1/6}<1$ and $2(C_1+C_v+1)\eps^{1/3}<1$.

At this point we observe that the length of every
curve of the triod (being all the curves graphs in the annulus
$B_{3M\sqrt{2(T-t_{j_M})}}(x_0)\setminus
B_{2M\sqrt{2(T-t_{j_M})}}(x_0)$) is bounded from below by a uniform factor
(depending only on the slope
$v$) times $M\sqrt{T-t_{j_M}}$. Then, by means of Lemma~\ref{kkevol}, we now prove
an inequality for the time derivative of the $L^2$ norm of
the curvature of the triods $\widehat{\TT}_t$ which are determined by the three
(moving in time) endpoints $Q^i(t)$, for $t\in [\widetilde t_{j_M},T)$.
Notice that here the constants $C_1$
and $C_2$ are ``universal'', $C_v$ depends only on the slope of the curve
with respect to the line $x_0+L^i$ and we use the two previous inequalities to
estimate the terms coming from the endpoints:
\begin{align*}
\frac{d}{dt} \int_{\widehat{\TT}_t} k^2 \,ds
\leq &\,C_1 \Big( \int_{\widehat{\TT}_t} k^2 \,ds
\Big)^3 +\frac{C_2C_v}{M\sqrt{T-t_{j_M}}}\Big(
\int_{\widehat{\TT}_t} k^2\, ds\Big)^2
+ \frac{C_v{\eps}^{1/2}}{(T-t_{j_M})^\frac 3 2}\\
\leq &\,C_1 \Big( \int_{\widehat{\TT}_t} k^2 \,ds
\Big)^3 + \frac{C_v \eps^\frac 1 {3}}{\sqrt{T-t_{j_M}}}\Big(
\int_{\widehat{\TT}_t} k^2\, ds\Big)^2
+ \frac{C_v{\eps}^{1/2}}{(T-t_{j_M})^\frac 3 2}\\
\leq &\,C_1 \Big( \int_{\widehat{\TT}_t} k^2 \,ds
\Big)^3 + \frac{\eps^\frac 1 {6}}{\sqrt{T-t_{j_M}}}\Big(
\int_{\widehat{\TT}_t} k^2\, ds\Big)^2
+ \frac{C_v{\eps}^{1/2}}{(T-t_{j_M})^\frac 3 2}\,,
\end{align*}
as we chose $M>C_2/{\eps^\frac 1 {3}}$ and $2C_v\eps^{1/6}<1$.\\
Then, letting
$$
A(t):= \max\left\{ \int_{\widehat{\TT}_t}k^2\,ds\,,\,\frac{\eps^\frac 1 {6}}{\sqrt{T-t_{j_M}}}\right\} \,,
$$
it follows
\begin{equation*}
A^\prime(t)\le \overline{C}_v A^3(t)
\end{equation*}
for almost every $t\in [\widetilde t_{j_M},T)$, where the constant $\overline{C}_v$ is given by $C_1+C_v+1$.\\
Integrating this differential inequality and recalling estimate~\eqref{eqeps1}, implying that 
$$
A(\widetilde{t}_{j_M})\leq
\max\left\{\frac{\sqrt{3}\eps}{\sqrt{2(T-t_{j_M})}}\,,\,\frac{\eps^\frac
    1 {6}}{\sqrt{T-t_{j_M}}}\right\}\leq\frac{\eps^\frac 1
  {6}}{\sqrt{T-t_{j_M}}}\,,
$$
as $\eps<1/2$, we get
\begin{equation*}
A(t)\le \frac{1}{\sqrt{A(\widetilde t_{j_M})^{-2}-2\overline{C}_v(t-\widetilde t_{j_M})}}\,,
\end{equation*}
hence,
\begin{equation*}
A(t)\le\frac{\eps^\frac 1{6}}{\sqrt{T-t_{j_M}
-2\overline{C}_v\eps^\frac 1 {3} (t-\widetilde t_{j_M})}}\,,
\end{equation*}
for every $t\in[\widetilde t_{j_M},T)$.\\
As $(t-\widetilde{t}_{j_M})\leq (T-t_{j_M})$, it follows that the
function $A(t)$ is uniformly bounded on $[\widetilde t_{j_M},T)$ as
soon as $2\overline{C}_v\eps^{\frac{1}{3}}<1$, which is satisfied by our
previous assumption on $\eps>0$.

We now notice that the three curves of the triod $\TT_t$, connecting
respectively the points $P^i$ and $Q^i$ (determined by
$\TT_t\setminus\widehat{\TT}_t$) cannot get too close
to the point $x_0=\lim_{t\to T}O(t)$ along the flow. Indeed, the parts of these
curves in the annulus 
$$
B_{5M\sqrt{2(T-t_{j_M})}}(x_0)\setminus
B_{3M\sqrt{2(T-t_{j_M})}}(x_0)
$$
are graphs for every 
$t\in[\widetilde t_{j_M},T)$, while the remaining pieces ``outside'' at
time $t=\widetilde{t}_{j_M}$, by maximum principle, during their
subsequent evolution can never get into the circle of radius
$R(t)=\sqrt{16M^2(T-t_{j_M})-2(t-t_{j_M})}$ and center $x_0$, also
moving by mean curvature in the time interval $[\widetilde t_{j_M},T)$
and, as $t\to T$, converging to the circle of radius
$$
\sqrt{16M^2(T-t_{j_M})-2(T-t_{j_M})}=\sqrt{(16M^2-2)(T-t_{j_M})}\,,
$$
which is clearly positive as $M^2>2$, hence far from the point $x_0$.\\
Consequently, since the closed subset of the set of reachable points obtained
as possible limit points of these three curves as
 $t\to T$ is contained in a closed set far from $x_0$, by
 Propositions~\ref{noline} and~\ref{nohalf}, we can
 cover such a set by a finite number of balls where the curvature of
 the evolving triod is uniformly bounded during the flow.
Being also the total length of the evolving triods uniformly bounded and being the $L^2$ norm of the curvature of the ``subtriods'' $\widehat{\TT}_t$,
given by the square root of the uniformly bounded function $A(t)$, we conclude that the full $L^2$ norm of the
curvature of the evolving triods ${\TT}_t$ is bounded, in
contradiction with Proposition~\ref{curvexplod}. This concludes the proof.
\end{proof}

\begin{Remark}
We point out that the ``regularity'' part of the main result of this paper, namely
Theorem~\ref{teomain}, can be extended with a similar proof to a triod
evolving by curvature with Neumann boundary conditions (the convergence statement does not hold in general, 
as in this case Steiner triods are unstable and possibly nonunique, see~\cite{freire1,garkoh,itokyan}). 
Moreover, whenever the classification given in
Proposition~\ref{resclimit} holds, the same proof also applies to the evolution of a network with multiple
triple junctions. For instance, this is true for a
network without loops and with at most {\em two} triple
junctions. Indeed, in this case, Proposition~\ref{dlteo} still holds 
and all the subsequent arguments can be adapted with minor
modifications.\\
In this respect, we take the occasion to underline a mistake
in~\cite[Remark 4.5]{mannovtor} (pointed out to us by T.~Ilmanen), 
where the authors claim that Proposition~\ref{dlteo} (Theorem~4.6 in~\cite{mannovtor}) holds for 
{\em any} network (without loops), {\em without any constraint on the number of
triple junctions}. Actually, the proof of Proposition~\ref{dlteo} can
be generalized only to networks in the plane with {\em at most two
  triple junctions}.
\end{Remark}

\bibliographystyle{amsplain}
\bibliography{biblio}

\end{document}